\documentclass[12pt,a4paper]{article}
\usepackage{amsmath,amssymb, amsthm}
\usepackage{lscape}
\usepackage{color}
\usepackage{array}
\usepackage{colortbl}
\usepackage{graphicx}
\usepackage{psfrag}
\usepackage{subfigure}
\definecolor{grey}{rgb}{0.7,0.7,0.7}
\newtheorem{theorem}{Theorem}[section]
\newtheorem{lemma}[theorem]{Lemma}
\newtheorem{corollary}[theorem]{Corollary}

\newtheorem{proposition}[theorem]{Proposition}

\begin{document}

\newcommand\msn{M_n(\bar{\mathbb{R}})}
\newcommand\mstwo{M_2(\bar{\mathbb{R}})}
\newcommand\barr{\bar{\mathbb{R}}}
\newcommand\hatr{\hat{\mathbb{R}}}
\newcommand\leqr{\leq_{\mathcal{R}}}
\newcommand\leql{\leq_{\mathcal{L}}}
\newcommand\leqj{\leq_{\mathcal{J}}}
\newcommand\twomat[4]{\left( \begin{array}{c c} #1 & #2  \\ #3 & #4 \\ \end{array} \right)}
\newcommand\ol[1]{\overline{#1}}
\newcommand\prjr[1]{PR(#1)}
\newcommand\prjc[1]{PC(#1)}

\begin{center}
{ \huge{Multiplicative structure of $2 \times 2$ tropical matrices}}

\bigskip

MARIANNE JOHNSON\footnote{Email \texttt{Marianne.Johnson@manchester.ac.uk}.
Research partially supported by the \textit{Manchester Centre for Interdisciplinary
Computational and Dynamical Analysis} (EPSRC grant EP/E050441/1).} and MARK
KAMBITES\footnote{Email \texttt{Mark.Kambites@manchester.ac.uk}. Research
supported by an RCUK Academic Fellowship.}

    \medskip

    School of Mathematics, \ University of Manchester, \\
    Manchester M13 9PL, \ England.

    \medskip

\end{center}

\begin{abstract}
We study the algebraic structure of the semigroup of all $2 \times 2$
tropical matrices under multiplication. Using ideas from tropical geometry,
we give a complete description of Green's relations and the idempotents
and maximal subgroups of this semigroup.
\end{abstract}

\section{Introduction}

Tropical algebra (also known as max-plus algebra) is the linear
algebra of the real numbers augmented with $-\infty$ when equipped
with the binary operations of addition and maximum. Interest in this
branch of mathematics is motivated by a wide range of applications
in numerous subject areas including combinatorial optimisation and
scheduling problems \cite{butkovic}, analysis of discrete event
systems \cite{maxplus}, control theory \cite{cohen}, formal language
and automata theory \cite{pin, simon}, phylogenetics
\cite{eriksson}, statistical inference \cite{pachter}, algebraic
geometry \cite{bergman, mikhalkin, richter} and
combinatorial/geometric group theory \cite{bieri}. Tropical algebra
and many of its basic properties have been independently
rediscovered many times by researchers in these fields. The first
detailed axiomatic study of ``max-plus algebra" was conducted by
Cuninghame-Green \cite{cuninghame} and this theory has been
developed further by a number of researchers (see \cite{baccelli,
heidergott} for surveys).

Many problems arising from application areas are naturally expressed
as tropical matrix algebra problems, and much of the theory of
tropical algebra is concerned with matrices. An important aspect is
the algebraic structure of tropical matrices under multiplication;
many authors have proved a number of interesting \textit{ad hoc}
results (see for example \cite{dalessandro, gaubert, pin, simon})
but so far there has been no systematic study in this area. This
surprising omission is due largely to the difficulty, both
conceptual and technical, of the subject. Even the case of $2 \times
2$ matrices, which is the main object of study in this paper,
demonstrates a number of interesting phenomena. We believe
that the development of a coherent and comprehensive theory of
tropical matrix semigroups of arbitrary finite dimension is a major
challenge.

The aim of this paper is to initiate the systematic study of the
semigroup-theoretic structure of tropical matrices under
multiplication, by considering the most natural starting
point: the monoid of all $2 \times 2$ tropical matrices. We give a
complete geometric description of Green's relations in this
semigroup, from which we are also able to deduce that the semigroup
is regular, and to describe all of its maximal subgroups. Since
conducting this research, we have learned that an independent study of
some of these topics has recently been conducted by Izhakian and
Margolis \cite{izhakiantalk}.

In addition to this introduction, this paper comprises three sections.
In Section~\ref{sec_prelim} we give a brief expository introduction to
the tropical semiring and tropical matrix algebra, including a summary
of known results about tropical matrix semigroups. Section~\ref{sec_green} is devoted to an
examination of the ideal structure of the monoid of all $2 \times 2$
tropical matrices, obtaining in particular geometric descriptions of Green's
relations $\mathcal{L}$, $\mathcal{R}$, $\mathcal{H}$, $\mathcal{D}$ and
$\mathcal{J}$, and of the associated partial orders.
Finally, in Section~\ref{sec_idpt} we consider the idempotent elements
of this monoid; combined with the results of the previous section, this
allows us to prove that the monoid is \textit{regular}, and to describe
completely its maximal subgroups.

\section{Preliminaries}\label{sec_prelim}

Let $\bar{\mathbb{R}} = \mathbb{R} \cup \{-\infty\}$. We extend the
addition and order on $\mathbb{R}$ to $\bar{\mathbb{R}}$ in
the obvious way, and define
operations multiplication $\otimes$ and addition $\oplus$ on $\bar{\mathbb{R}}$
by $a \otimes b = a + b$ and $a \oplus b = \max\{a, b\}$ for all
$a, b \in \bar{\mathbb{R}}$. Then $\bar{\mathbb{R}}$ is a
semiring with multiplicative identity $0$ and additive identity
$-\infty$. In fact, $\bar{\mathbb{R}}$ is an idempotent semifield,
since 
$a \oplus a = a$ for all $a \in \bar{\mathbb{R}}$ and $a \otimes -a
= 0$ for all $a \in \mathbb{R}$. We call $(\bar{\mathbb{R}},
\otimes, \oplus )$ the \emph{tropical semiring}; some authors refer to
it as the \textit{max-plus} semiring.

For each positive integer $n$ let $M_n(\bar{\mathbb{R}})$ denote the
set of $n \times n$ matrices with entries in $\bar{\mathbb{R}}$. The
$\otimes$ and $\oplus$ operations on $\barr$ induce corresponding
operations on $M_n(\bar{\mathbb{R}})$ in the obvious way. Indeed, if
$A, B \in M_n(\bar{\mathbb{R}})$ then we have
\begin{eqnarray*}
(A \otimes B)_{ij} &=& \bigoplus_{k=1}^{n} A_{ik} \otimes B_{kj}, \text{ and }\\
(A \oplus B)_{ij} &=& A_{ij} \oplus B_{ij},
\end{eqnarray*}
for all $1 \leq i, j \leq n$, where $X_{i,j}$ denotes the $(i,j)$th
entry of the matrix $X$. For brevity, we shall usually write
$AB$ in place of $A \otimes B$ for a product of matrices. It is then
easy to check that $M_n(\bar{\mathbb{R}})$ is an idempotent
semiring, with multiplicative identity
\[
\left(
\begin{array}{ccccc}
0&-\infty &\cdots& -\infty\\
-\infty&0 & \ddots &\vdots\\
\vdots&\ddots &\ddots &-\infty\\
-\infty&\cdots &-\infty &0\\
\end{array}\right)
\]
and additive identity
\[
\left(
\begin{array}{ccc}
-\infty&\cdots& -\infty\\
\vdots& \ddots&\vdots\\
-\infty&\cdots &-\infty\\
\end{array}\right).
\]
We call $(M_n(\bar{\mathbb{R}}), \otimes, \oplus)$ the \emph{$n
\times n$ tropical matrix semiring}. The main object of study in this
paper is the multiplicative monoid of this semiring, which we shall
refer to simply as $M_n(\bar{\mathbb{R}}), \otimes$.

We summarise some known results about this semigroup. It is readily
verified (see for example \cite{ellis}) that the invertible elements
of $\msn$ (the \textit{units} in
the terminology of ring theory or semigroup theory) are exactly the
\textit{monomial matrices}, that is, matrices with exactly one entry
in each row and column not equal to $-\infty$. It follows easily that
the group of units in $\msn$ is isomorphic to the permutation wreath
product $\mathbb{R} \wr (S_n, \lbrace 1, \dots, n \rbrace)$ of the additive
group of real numbers with the symmetric group on $n$ points.

It is known \cite{dalessandro} that the semigroup $\msn$ is
\textit{weakly permutable}, in the sense that there is a positive
integer $k$ such that every sequence of $k$ elements admits two
distinct permutations such that the corresponding products of
elements are equal in the semigroup. It is clear from the definition
that weak permutability is inherited by subsemigroups. It is also
known \cite{blyth,curzio} that a group is weakly permutable if and
only if it has an abelian subgroup of finite index. It follows
that every subgroup of $\msn$ (including those whose identity
element is an idempotent other than the identity of $\msn$) has an
abelian subgroup of finite index. Moreover, it is also shown in
\cite{dalessandro} that finitely generated subsemigroups of $\msn$
have polynomial growth.

The semigroup $M_n(\bar{\mathbb{R}})$ acts naturally on the left and
right of the space of $n$-vectors over $\barr$, known as
\textit{affine tropical $n$-space}. Notice that a tropical multiple
of a vector $(x_1, \ldots, x_n) \in \bar{\mathbb{R}}^n$ has the form
$(x_1+\lambda, \ldots, x_n+\lambda)$ for some $\lambda \in
\bar{\mathbb{R}}$. From affine tropical $n$-space we obtain
\textit{projective tropical $(n-1)$-space} by discarding the zero
vector $(-\infty,
 \ldots, -\infty)$ and identifying two non-zero vectors if
one is a tropical multiple of the other.

We can represent affine tropical $2$-space (or the \textit{tropical
plane}) pictorially as a quadrant of the Euclidean plane with two
sets of axes as shown in Figure \ref{axes}. The set of tropical
multiples of $v \in \barr^2$ is then equal to the line of gradient 1
which passes through $v$, as shown in Figure~\ref{scaling}; notice that this
line includes the zero vector. Vector addition in $\barr^2$ may
also be described pictorially as follows. For $u, v \in \barr^2$ the
sum $u \oplus v$ is given by the upper right-most vertex of the
unique rectangle with $u$ and $v$ as vertices and edges parallel to
the axes, see Figure~\ref{addition}. Note that the sides of this
rectangle may have infinite length.

\begin{figure}[htp]
\begin{center}
\includegraphics[scale=0.5]{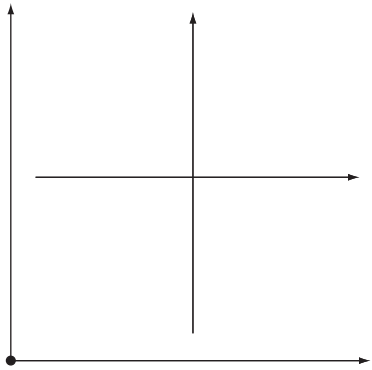}\caption{The tropical axes.}\label{axes}
\end{center}
\end{figure}

\begin{figure}[ht]
\begin{center}\psfrag{a}{u}\psfrag{b}{$u\oplus v$ }\psfrag{c}{$v$}
\psfrag{d}{\hspace{-0.7cm}$\lambda \otimes v$}\psfrag{e}{$v$}
\subfigure[{Tropical vector scaling in $\barr^2$.
}]{\includegraphics[scale=0.5]{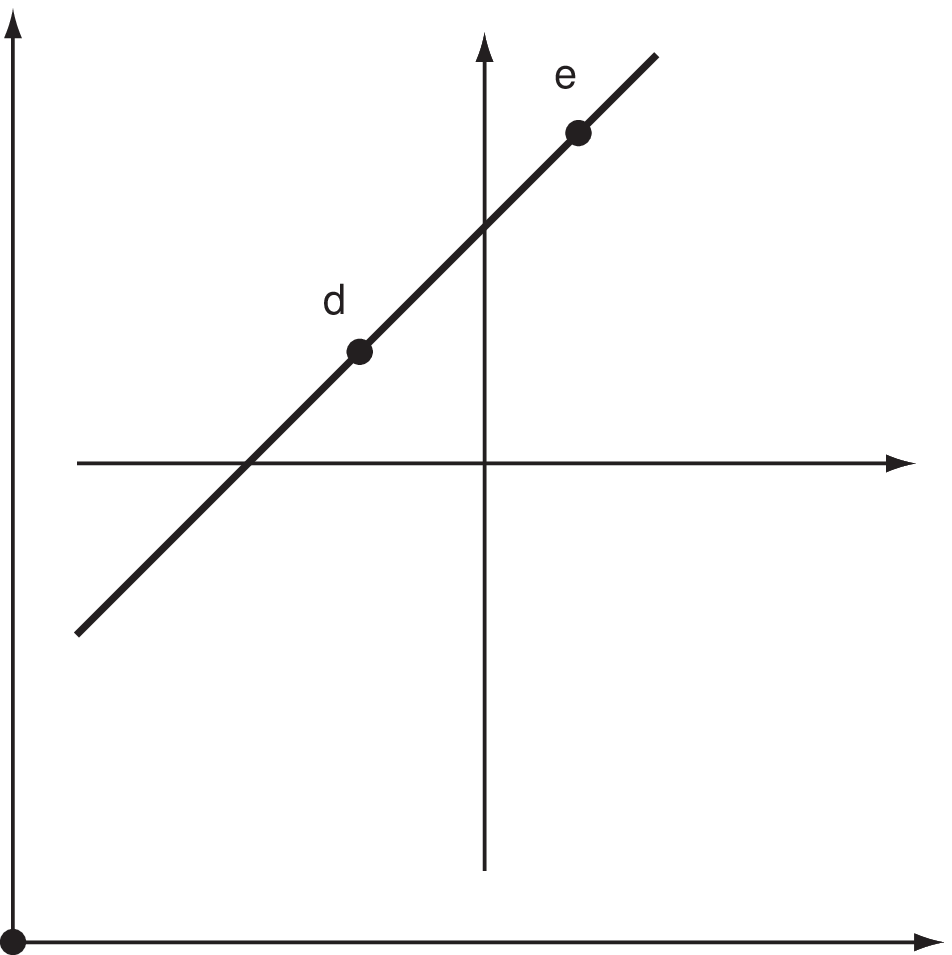}\label{scaling}}\hspace{1cm}
\subfigure[{Tropical vector addition in
$\barr^2$.}]{\includegraphics[scale=0.5]{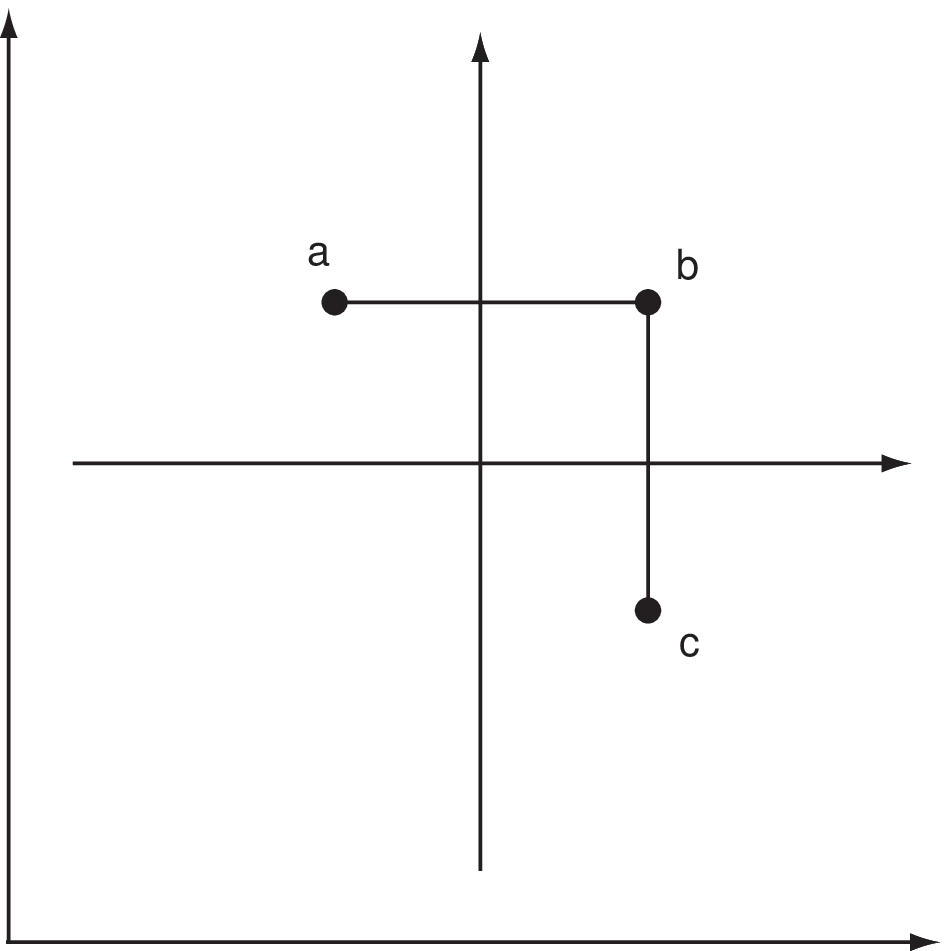}\label{addition}}
\caption{Tropical linear combinations of
vectors.}\label{lincomb}\end{center}
\end{figure}

\textit{Projective} tropical $1$-space can be conveniently identified with
the two point compactification of the real line
$$\hatr = \mathbb{R} \cup \lbrace -\infty, \infty \rbrace$$
via the map which takes the equivalence class of a non-zero vector
$(a,b) \in \barr^2$ to $b-a$ if $a$ and $b$ are real, $\infty$ if $a
= -\infty$ and $-\infty$ if $b = -\infty$\footnote{In fact, if we
extend subtraction in the obvious way to $\bar{\mathbb{R}} \times
\bar{\mathbb{R}} \setminus \lbrace (-\infty, -\infty) \rbrace$,
we have that the projection of $(a,b)$
corresponds to $b-a$ for all non-zero points $(a,b)$.}. In
pictorial terms, the image of a point $(a,b)$ with real coordinates
under this projection may be thought of as the intercept of the line
of gradient $1$ through the point $(a,b)$ with the
vertical\footnote{The choice of the vertical axis here is
of course arbitrary. One could instead take signed perpendicular distance
of the given line from the point $(0,0)$; this is arguably
conceptually cleaner but makes no practical difference and
introduces an extra factor of $\sqrt{2}$ into computations.} axis
through $(0,0)$.

\section{Green's Relations}\label{sec_green}

We begin by briefly recalling the definitions of a number of binary
relations which are used to analyse the structure of a monoid. For
further reference and examples we refer the reader to \cite{clifford}.

Let $S$ be a monoid and let $A, B \in S$.
We define a binary relation $\leqr$ on $S$ by $A \leqr B$ exactly if
$AS \subseteq BS$, or equivalently, if $A = BX$ for some $X \in S$.
Similarly, we define $A \leql B$ if $SA \subseteq SB$, and $A \leqj
B$ if $SAS \subseteq SBS$. The relations $\leqr$, $\leql$ and
$\leqj$ are \textit{preorders} (reflexive, transitive binary
relations) on the monoid $S$.

Next, we define a binary relation $\mathcal{R}$ on $S$ by
$A\mathcal{R}B$ if $A$ and $B$ generate the same principal right
ideal in $S$, or equivalently, if $A \leqr B$ and $B \leqr A$.
Similarly, we define $A\mathcal{L}B$ if $A$ and $B$ generate the
same principal left ideal in $S$, and $A \mathcal{J} B$ if $A$ and
$B$ generate the same principal two-sided ideal in $S$. The
relations $\mathcal{R}$, $\mathcal{L}$ and $\mathcal{J}$ are all
equivalence relations. In fact they are the largest equivalence
relations contained in the preorders $\leqr$, $\leql$ and $\leqj$
respectively, from which it follows that these preorders induce
partial orders on the equivalence classes of the respective
equivalence relations.

We let $\mathcal{H}$ denote the intersection $\mathcal{L} \cap
\mathcal{R}$, and $\mathcal{D}$ be the intersection of all
equivalence relations containing $\mathcal{L}$ and $\mathcal{R}$.
Both are equivalence relations, and it is well known and easy to
show that we have $A \mathcal{D} B$ if and only if there exists $Z
\in S$ such that $A \mathcal{R} Z$ and $Z \mathcal{L} B$.

We shall also need some basic ideas from tropical geometry. For each
positive integer $k$ we define a \emph{($k$-generated)
convex cone} in $\bar{\mathbb{R}}^n$ to be a non-empty set which is
the set of all tropical linear combinations of vectors from some given subset (of cardinality $k$
or less) of $\bar{\mathbb{R}}^n$. Convex cones are the tropical analogue of
linear subspaces in classical linear algebra. However, we shall
refrain from terming them \textit{(tropical linear) subspaces},
since this term is generally applied to a distinct concept which in
tropical geometry plays the role of affine linear subspaces in
classical algebraic geometry \cite{develin}.

Since convex cones are closed under scaling, each convex cone $V$ in 
affine tropical $n$-space is naturally associated with a subset in 
projective $(n-1)$-space, which we call the \textit{projectivisation} of 
$V$. We define a \textit{($k$-generated) convex set} in projective 
tropical $(n-1)$-space to be the projectivisation of a ($k$-generated) 
convex cone in affine tropical $n$-space. In the case that $n=2$, so that 
the projective space is $\hatr$, it is easily seen that the only convex 
sets are the empty set, the singleton sets and \textit{intervals} (open, closed,
half-open and half-closed) where the latter are defined in the obvious
way using the order on $\hatr$. The $2$-generated convex sets are the
empty set, singleton sets, and \textit{closed} intervals of $\hatr$; we
call these the \textit{closed} convex sets.

Now let $A \in \msn$. We define the \textit{column
space} $C(A)$ of $A$ to be the convex cone which is the set of tropical
linear combinations of the columns of $A$.
We shall also be interested in the projectivisation of $C(A)$, which
we call the \textit{projective column space} of $A$ and denote
$\prjc{A}$. Dually, the \textit{row space} $R(A)$ of $A$ is the
convex cone given by the set of tropical linear combinations of the
rows of $A$, and its projectivisation is called the
\textit{projective row space} of $A$, denoted $\prjr{A}$.

The following characterisation of the $\mathcal{R}$ and
$\mathcal{L}$ preorders is well known at least in the
case of matrices over fields (see for example
\cite[Lemma~2.1]{okninski}) and extends without difficulty to matrices
over the tropical semiring. For completeness, we include a brief proof.

\begin{lemma}\label{lemma_rclasses}
Let $A, B \in M_n(\bar{\mathbb{R}})$. Then the following are
equivalent:
\begin{itemize}
\item[(i)] $A \leqr B$ [respectively, $A \leql B$];
\item[(ii)] $C(A) \subseteq C(B)$ [respectively, $R(A) \subseteq R(B)$] in affine tropical $n$-space;
\item[(iii)] $\prjc{A} \subseteq \prjc{B}$ [respectively, $\prjr{A} \subseteq \prjr{B}$] in projective tropical $(n-1)$-space.
\end{itemize}
\end{lemma}
\begin{proof}
We prove the equivalence of the statements involving $\leqr$ and
column spaces, the equivalence of the statements involving $\leql$
and row spaces being dual. The equivalence of (ii) and (iii)
follows from the fact that convex cones, and hence column spaces, are closed
under scaling, so it will suffice to show that (i) and (ii) are equivalent.

If (i) holds, that is, $A \leqr B$, then by
definition there is a matrix $X \in \msn$ such that $BX = A$. Now, since
the columns of $BX$ are contained in $C(B)$ it follows that
 $C(BX) = C(A) \subseteq C(B)$ so that (ii) holds.
Conversely, suppose that (ii) holds. Since the tropical semiring has a
multiplicative identity, the columns of $A$ are contained in $C(A)$, and
hence in $C(B)$. Thus, every column of $A$ can be written as a
linear combination of the columns of $B$, which means exactly that there
exists $X \in \msn$ such that $A = BX$. Thus (i) holds.
\end{proof}

\begin{corollary}\label{cor_rclasses}
Let $A, B \in M_n(\bar{\mathbb{R}})$. Then the following are
equivalent:
\begin{itemize}
\item[(i)] $A \mathcal{R} B$ [respectively, $A \mathcal{L} B$];
\item[(ii)] $C(A) = C(B)$ [respectively, $R(A) = R(B)$] in affine tropical $n$-space;
\item[(iii)] $\prjc{A} = \prjc{B}$ [respectively, $\prjr{A} = \prjr{B}$] in projective tropical $(n-1)$-space.
\end{itemize}
\end{corollary}

By Corollary~\ref{cor_rclasses}, the $\mathcal{R}$-classes
of $\mstwo$ are in a natural bijective correspondence with the $2$-generated
tropical convex cones in the tropical plane, and hence also with the
with the closed convex sets
in $\hatr$. For such set $M \subseteq \hatr$ we
denote by $R_M$ the corresponding $\mathcal{R}$-class. Since $\hatr$ with
the obvious topology is homeomorphic to the closed unit interval, and
the closed intervals are definable topologically, combining with
Lemma~\ref{lemma_rclasses} yields the following natural description
of the natural partial order on the $\mathcal{R}$-classes, or equivalently,
on the intersection lattice of principal right ideals.

\begin{corollary}
The lattices of principal right ideals and of principal left ideals in
$\mstwo$ are isomorphic to the intersection lattice generated by the
closed subintervals of the closed unit interval.
\end{corollary}

It follows easily from the description of tropical vector scaling
and addition given in Section~\ref{sec_prelim} that the
$2$-generated convex cones in the affine tropical plane can take $8$
essentially distinct forms. Figure~\ref{fig:cones} shows these in
affine space, the captions giving the associated subsets of projective
space $\hatr$.

\begin{figure}[ht]
\begin{center}
\psfrag{g}{\hspace{-0.2cm}$y$}
\psfrag{f}{\hspace{-0.2cm}$y$}\psfrag{i}{\hspace{-0.2cm}$y$}\psfrag{h}{$x$}\psfrag{j}{\hspace{0.1cm}$y$}
\subfigure[$\emptyset$]{\includegraphics[scale=0.25]{dot.eps}\label{emptyset}}\hspace{1cm}
\subfigure[$\lbrace -\infty
\rbrace$]{\includegraphics[scale=0.25]{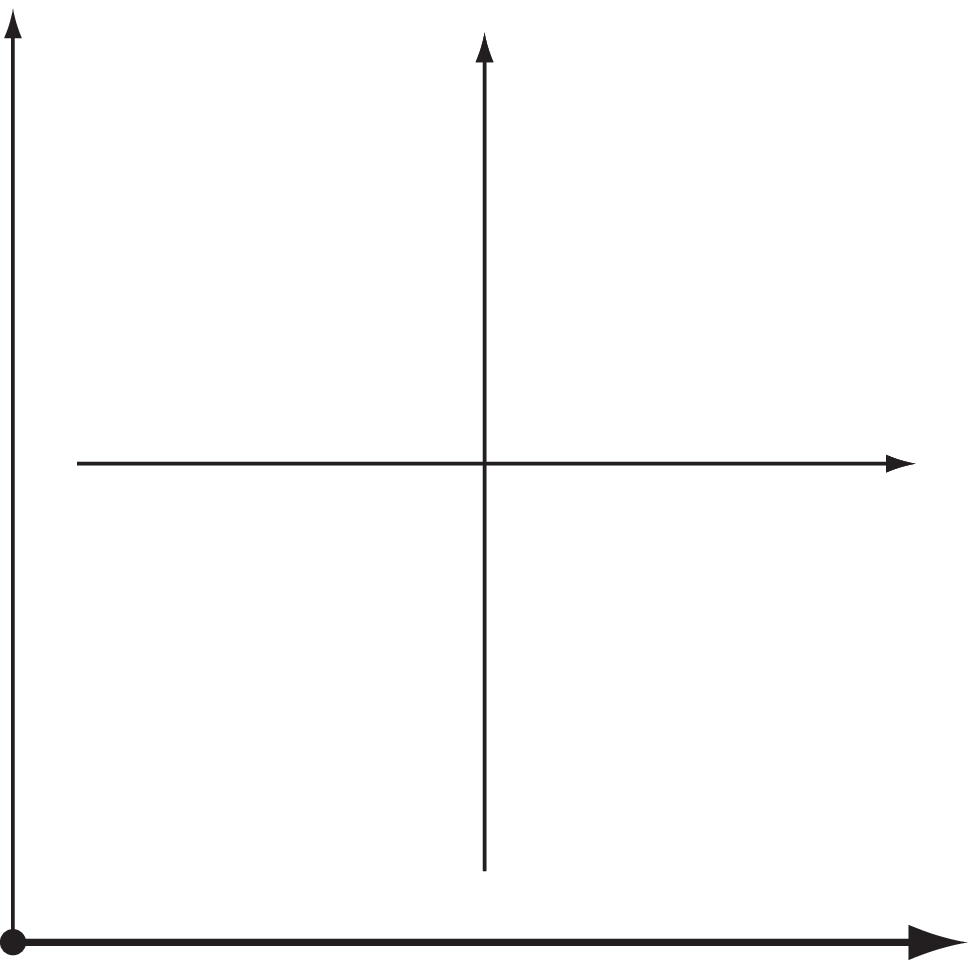}\label{minusinfty}}
\hspace{1cm} \subfigure[$\lbrace y
\rbrace$]{\includegraphics[scale=0.25]{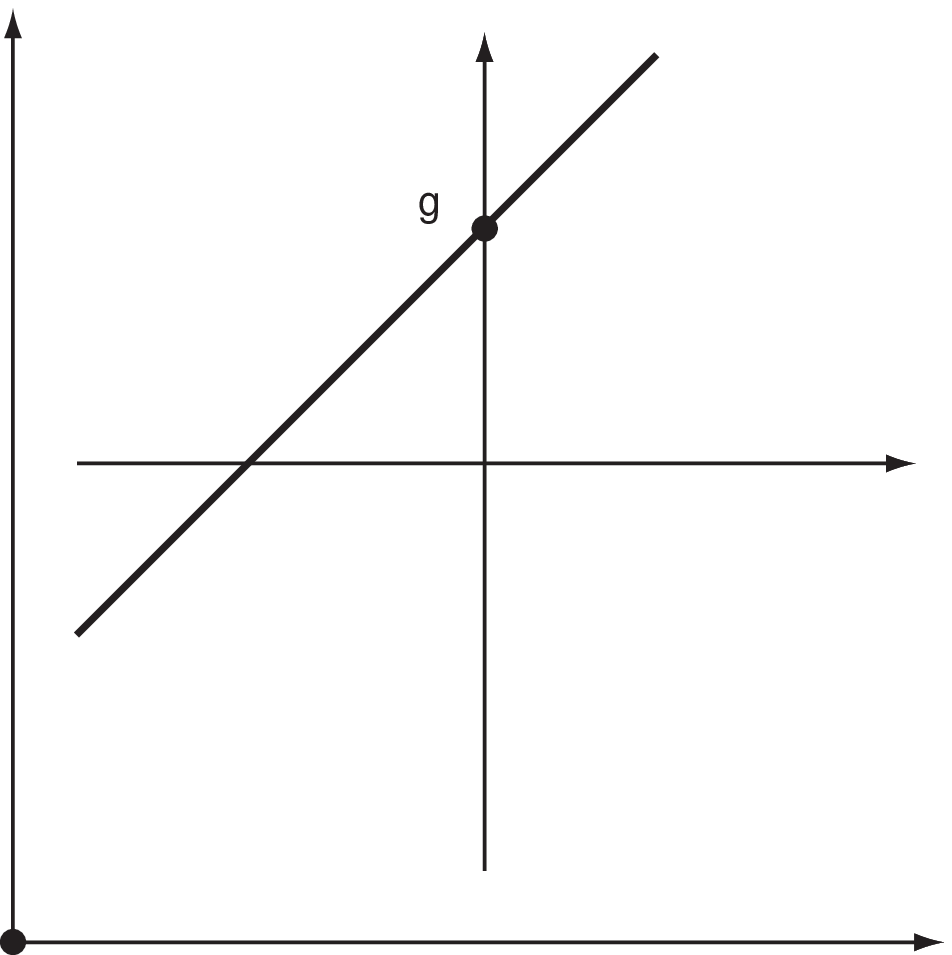}\label{k}}
\hspace{1cm} \subfigure[$\lbrace \infty
\rbrace$]{\includegraphics[scale=0.25]{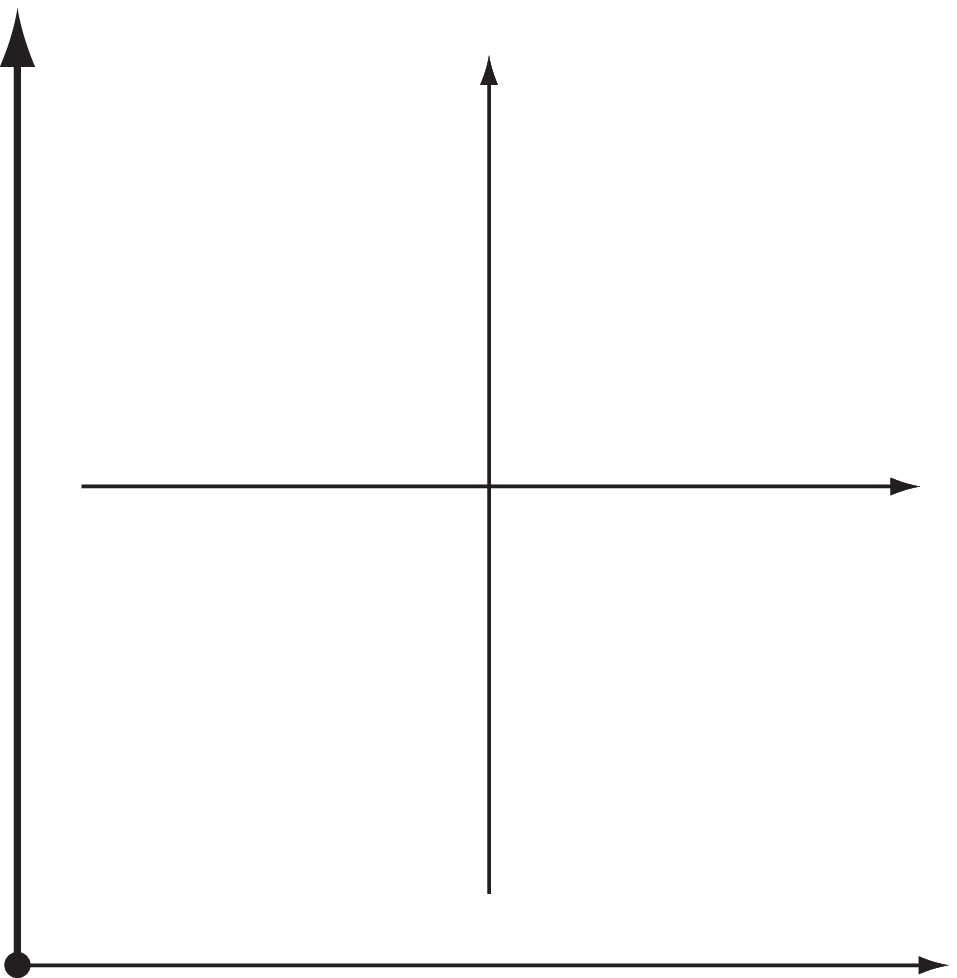}\label{infty}}
\hspace{1cm} \subfigure[${\left[-\infty,
y\right]}$]{\includegraphics[scale=0.25]{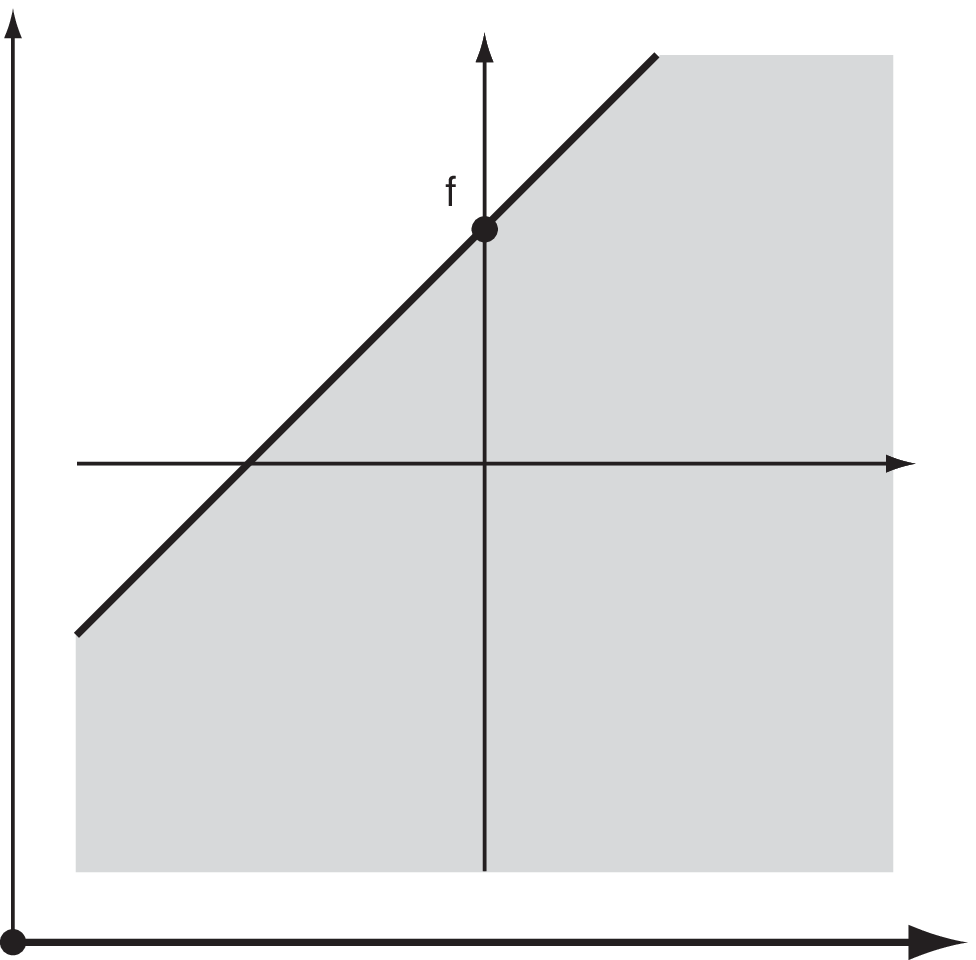}\label{minusinftyk}}
\hspace{1cm}
\subfigure[${\left[x,y\right]}$]{\includegraphics[scale=0.25]{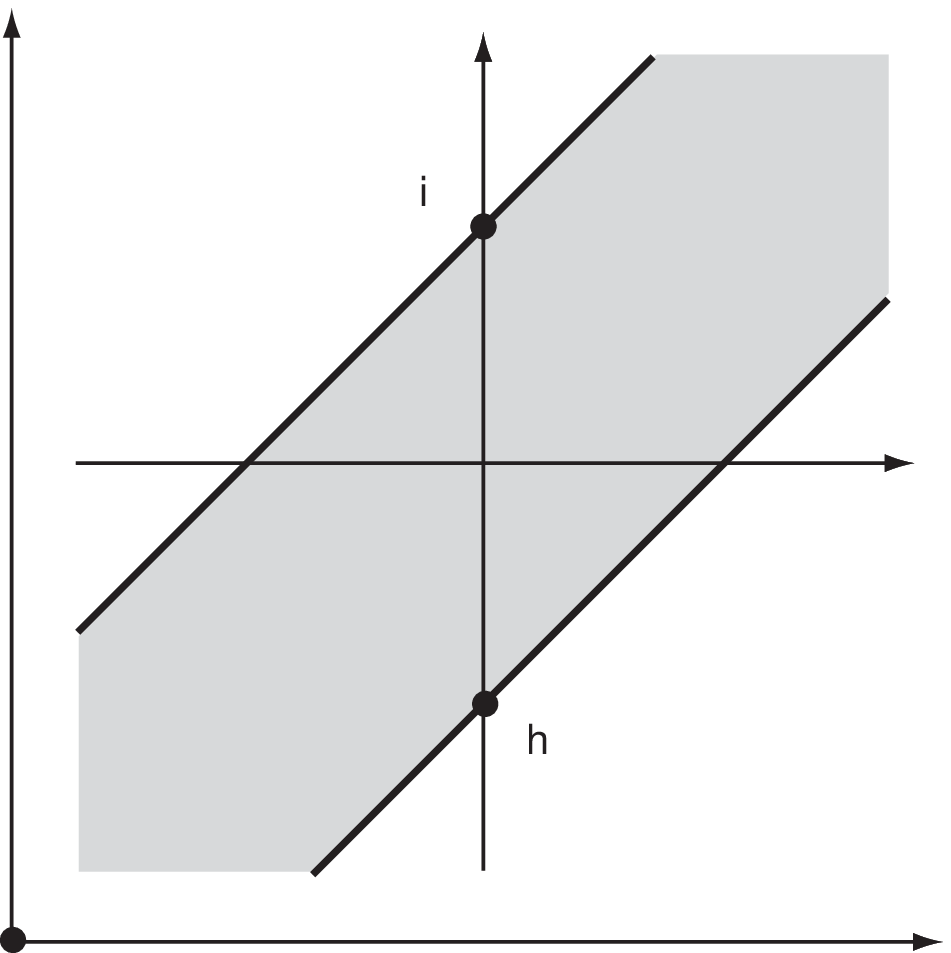}\label{jk}}
\hspace{1cm} \subfigure[${\left[y,
\infty\right]}$]{\includegraphics[scale=0.25]{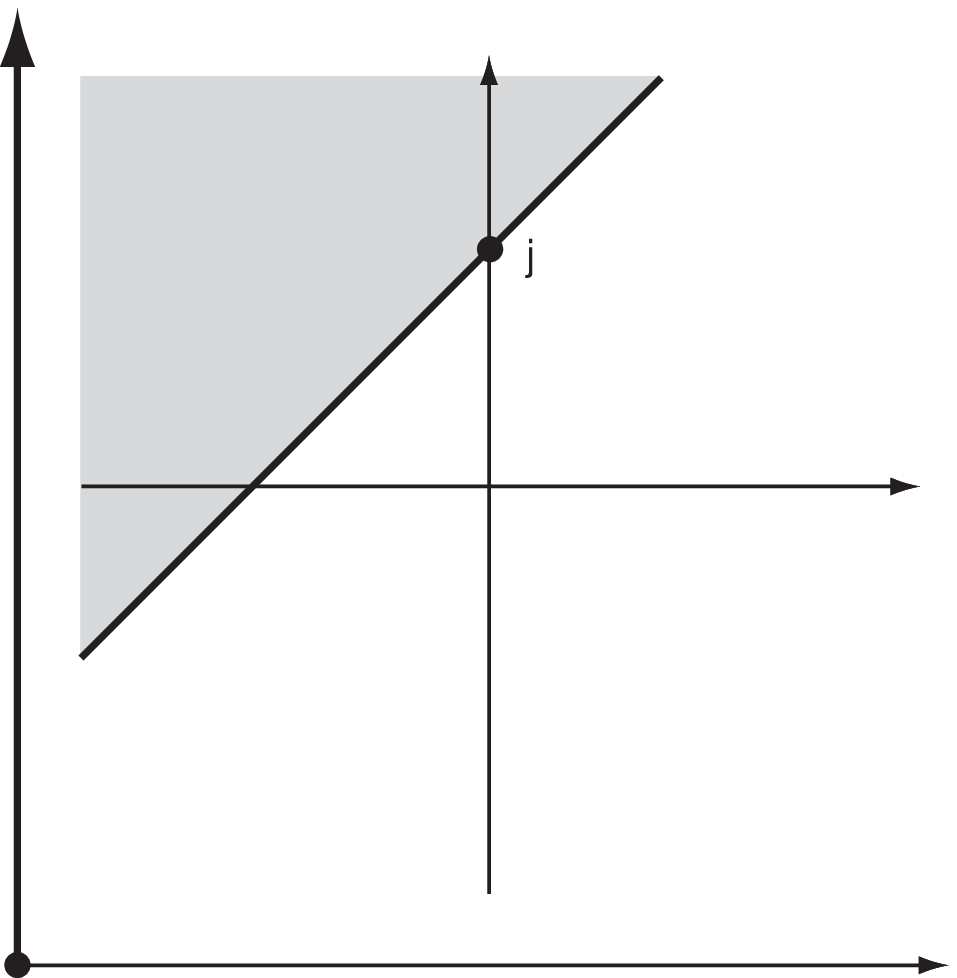}\label{kinfty}}
\hspace{1cm} \subfigure[${\left[-\infty,
\infty\right]}$]{\includegraphics[scale=0.25]{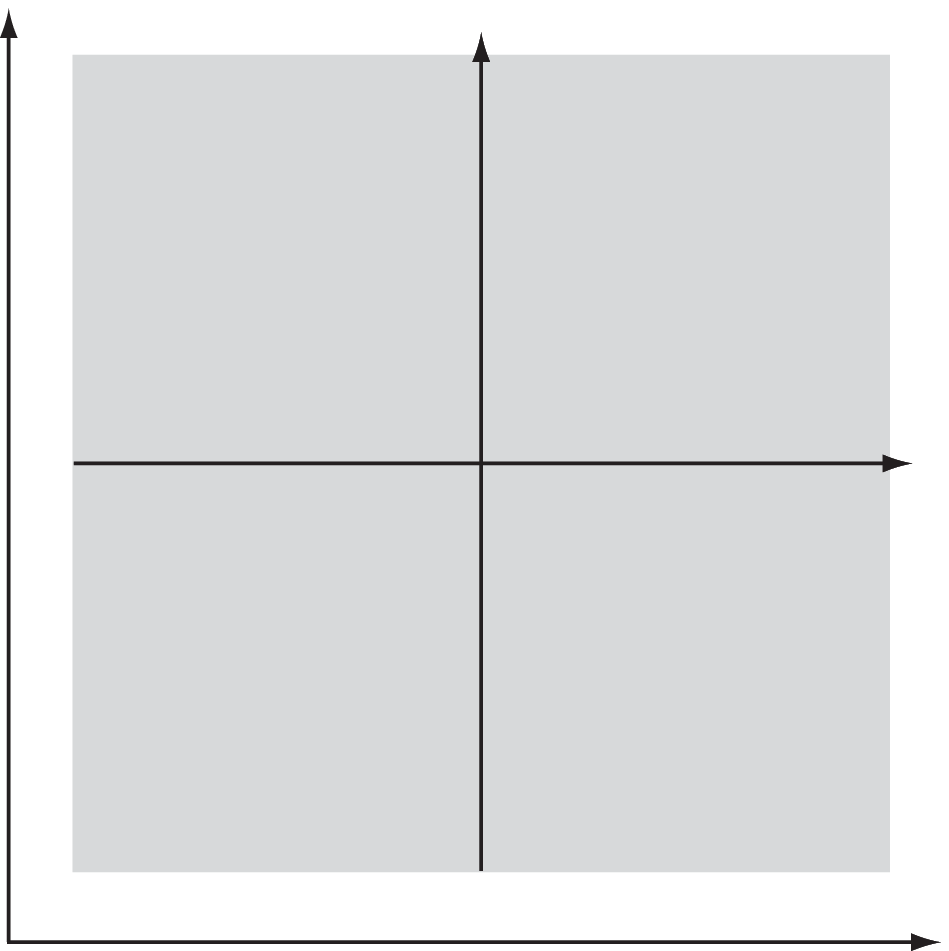}\label{hatr}}
\caption{The $2$-generated tropical convex cones of
$\bar{\mathbb{R}}^2$, which correspond to closed intervals in
$\hatr$ and to the $\mathcal{R}$-classes of
$M_2(\bar{\mathbb{R}})$.}\label{fig:cones}\end{center}
\end{figure}

Using the geometric description of tropical vector operations given
in Figure~\ref{lincomb}, it is easily seen that for a non-zero
matrix
$$A = \left( \begin{array}{c c} a & b  \\ c&d\\ \end{array} \right)$$
the (affine) column space $C(A)$ is exactly the region of the
quadrant bounded by the lines
$$\lbrace (a+\lambda, c+\lambda) \mid \lambda \in \barr \rbrace \text{ and } \lbrace (b+\lambda, d+\lambda) \mid \lambda \in \barr \rbrace.$$
If  $A$ has a zero column, $a=c=-\infty$, say, then the projective
column space of $A$ is  the singleton $\{d-b\}$ (using the natural
extension of substraction to $\barr \times \barr \setminus \lbrace (-\infty,-\infty) \rbrace$
as described in Section~\ref{sec_prelim}).
Otherwise, the projective
column space of $A$ is the closed interval (or singleton if $c-a = d-b$)
with endpoints
$c-a$ and $d-b$. Explicit descriptions of the $\mathcal{R}$-classes as sets of
matrices are given in Figure~\ref{fig:rclass}.

\begin{figure}[ht]
\begin{tabular}{l l l l}
$R_\emptyset$ &=& $\left\{
\left( \begin{array}{c c} -\infty & -\infty  \\
-\infty&-\infty\\
\end{array}
\right)\right\}$,\\
\vspace{-0.3cm}\\ $R_{\lbrace -\infty \rbrace}$ &=& $\left\{
\left( \begin{array}{c c} a& b  \\
-\infty&-\infty\\
\end{array}
\right)\mid a, b \in \bar{\mathbb{R}}, a \oplus b \in \mathbb{R}\right\}$,\\
\vspace{-0.3cm}\\
$R_{\lbrace y \rbrace}$ &=& $\left\{
\left( \begin{array}{c c} a & b  \\
a+y&b+y\\
\end{array}
\right)\mid a, b \in \bar{\mathbb{R}}, a \oplus b \in \mathbb{R}\right\}$,\\
\vspace{-0.3cm}\\
$R_{\lbrace \infty \rbrace}$ &=& $\left\{
\left( \begin{array}{c c} -\infty &-\infty  \\
a&b\\
\end{array}
\right)\mid a, b \in \bar{\mathbb{R}}, a \oplus b \in \mathbb{R}\right\}$,\\
\vspace{-0.3cm}\\
$R_{[-\infty, y]}$ &=& $\left\{\left( \begin{array}{c c} a &b \\
-\infty&b+y\\
\end{array}
\right),
\left( \begin{array}{c c} b & a  \\
b+y&-\infty\\
\end{array}
\right)\mid a, b \in \mathbb{R}\right\}$,\\
\vspace{-0.3cm}\\
$R_{[x,y]}$ &=& $\left\{\left( \begin{array}{c c} a &b \\
a+x&b+y\\
\end{array}
\right),
\left( \begin{array}{c c} b & a  \\
b+y&a+x\\
\end{array}
\right)\mid a, b \in \mathbb{R}\right\}$,\\
\vspace{-0.3cm}\\
$R_{[y,\infty]}$ &=& $\left\{\left( \begin{array}{c c} -\infty&a \\
b&a+y\\
\end{array}
\right),
\left( \begin{array}{c c} a & -\infty  \\
a+y&b\\
\end{array}
\right)\mid a, b \in \mathbb{R}\right\}$,\\
\vspace{-0.3cm}\\
$R_{\hatr}$ &=& $\left\{\left( \begin{array}{c c} a &-\infty  \\
-\infty&b\\
\end{array}
\right),
\left( \begin{array}{c c} -\infty & a  \\
b&-\infty\\
\end{array}
\right)\mid a, b \in \mathbb{R}\right\}$.\\
\end{tabular}
\caption{The $\mathcal{R}$-classes of $M_2(\bar{\mathbb{R}})$. The
parameters $x$ and $y$ run through all values in $\mathbb{R}$ with
$x<y$.}\label{fig:rclass}
\end{figure}

For $U \subseteq M_2(\bar{\mathbb{R}})$ we define the transpose of
$U$ to be the set $U^T$ of all transposes of matrices in $U$, $ U^T
= \{ A^T :A \in U\}$. It follows easily from
Corollary~\ref{cor_rclasses} that each $\mathcal{L}$-class is the
transpose of an $\mathcal{R}$-class; for each closed convex
subset $M$ of $\hatr$ we therefore define $L_M = R_M^T$.

Our next objective is to describe the $\mathcal{D}$ and
$\mathcal{J}$ relations and the $\mathcal{J}$-preorder on $\mstwo$.
Recall that every $\mathcal{D}$-class and every $\mathcal{J}$-class
is a union of $\mathcal{R}$-classes, and that the
$\mathcal{R}$-class of a matrix is determined by its projective
column space. It therefore follows that the $\mathcal{D}$ and
$\mathcal{J}$ relations can be described in terms of projective
column spaces (or symmetrically, of projective row spaces).
To obtain such a description, we consider the natural distance function
$\delta : \hatr \times \hatr \to \mathbb{R} \cup \lbrace \infty
\rbrace$ defined by
$$\delta(x,y) = \begin{cases}
|y-x| & \text{ if } x, y \in \mathbb{R} \\
0 & \text{ if } x = y = -\infty \text{ or } x = y = \infty \\
\infty & \text{ otherwise.}
\end{cases}$$
The function $\delta$ satisfies $\delta(x,y) = 0$ if and only if $x
= y$. It is also symmetric and satisfies a triangle inequality when
the usual order on $\mathbb{R}$ is extended to $\mathbb{R} \cup
\lbrace \infty \rbrace$ in the obvious way. It is thus a metric,
except that it may take the value $\infty$, and so induces obvious
notions of \textit{isometric embedding} and \textit{isometry}
between subsets of $\hatr$. For $M, N \subseteq \hatr$ we write $M
\cong N$ to denote that $M$ and $N$ are isometric. Note that we do
\textit{not} require isometries or isometric embeddings to preserve
the orientation of $\hatr$, so for example $[-\infty, 0] \cong [0, \infty]$.

We define the \textit{diameter} $d(S)$ of a subset $S \subseteq \hatr$ (or of an
isometry type of subsets of $\hatr$) to be
$$d(S) = \sup_{x,y \in S} \delta(x,y)$$
where of course $0$ is the supremum of the empty set, and $\infty$ the supremum of
any set not bounded above by a real number.

We shall be
particularly interested in isometries and isometric embeddings
between closed convex subsets of $\hatr$, where a simple
combinatorial characterisation applies.
It is readily verified that
two distinct such sets are isometric if and only if (i) they are
both singletons, (ii) they are both closed intervals of
the same finite diameter, or (iii) they are both closed intervals with one real
endpoint and one endpoint at $\infty$ or $-\infty$. It is also easy to
check that isometric embedding induces a partial order on the closed
convex subsets (the only non-trivial part of this claim being that the order
is antisymmetric, that is, that two such sets which embed isometrically into
each other are necessarily isometric).

\begin{proposition}\label{prop_duality}
Let $A \in \mstwo$. Then $\prjc{A} \cong \prjr{A}$.
\end{proposition}
\begin{proof}
We proceed by case analysis, considering each possible form of
$\prjc{A}$. If $\prjc{A} = \emptyset$ then $A$ is the zero matrix so
$\prjr{A} = \emptyset$. If $\prjc{A} = \hatr$ then $A$ is a unit and
so $\prjr{X} = \hatr$.

If $\prjc{A} = \lbrace y \rbrace$ is a singleton then $A \in
R_{\lbrace y \rbrace}$ for some $y \in \hatr$. By reference to
Figure~\ref{fig:rclass} we see that $A$ has at least one non-zero
row $(a,b)$. It is then easy to verify (for example, by locating $A^T$
in Figure~\ref{fig:rclass}) that in each case $A^T \in
R_{\lbrace b-a \rbrace}$, where we again using the extended subtraction.
defined in Section~\ref{sec_prelim}. Thus, $PR(A) = PC(A^T) = \lbrace b-a \rbrace$
is isometric to $PC(A)$.

If $\prjc{A} = [x,y]$ is a closed interval with real endpoints then
using Figure~\ref{fig:rclass} once again we see that either
$$A = \twomat{a}{b}{a+x}{b+y} \text{ or } A = \twomat{b}{a}{b+y}{a+x},$$
where $a,b \in \mathbb{R}$. In the former case we have
$$A^T = \twomat{a}{(a+x)}{a+(b-a)}{(a+x)+(b-a+y-x)}$$
from which it follows that $A^T \in R_{[b-a,b-a+y-x]}$ and $\prjr{A}
= [b-a,b-a+y-x]$ is again a closed interval of diameter $y-x$ and hence
isometric to $\prjc{A}$. The latter case is similar, as are the
cases where one end of the interval is $\infty$ or $-\infty$.
\end{proof}

\begin{proposition}\label{prop_intersect}
Let $M$ and $N$ be closed convex subsets in $\hatr$, and suppose
$M \cong N$. Then there exists a matrix $Z \in \mstwo$ such that
$\prjc{Z} = M$ and $\prjr{Z} = N$
\end{proposition}
\begin{proof}
Once again, the proof is by case analysis with reference to
Figure~\ref{fig:rclass}. If $M = \emptyset$ then $N = \emptyset$ and
it suffices to take $Z$ to be the zero matrix, while if $M = \hatr$
then $N = \hatr$ and we may take $Z$ to be the identity matrix.

Suppose now that $M = \lbrace x \rbrace$ is a singleton (with $x \in
\hatr$ either real or infinite). Then $N = \lbrace y \rbrace$ must
be a singleton too and by reference to Figure~\ref{fig:rclass} it is
seen that the matrices
$$A = \twomat{0}{y}{x}{x+y} \mbox{and } B = \twomat{-(x+y)}{-x}{-y}{0}$$
satisfy $A \in R_{\lbrace x \rbrace}$, $A^T \in R_{\lbrace y
\rbrace}$, for $x,y \neq \infty$ and $B \in R_{\lbrace x \rbrace}$,
$B^T \in R_{\lbrace y \rbrace}$, for $x,y \neq -\infty$. Similarly,
the matrices
$$X =\twomat{-\infty}{-\infty}{0}{-\infty}, Y =\twomat{-\infty}{0}{-\infty}{-\infty}$$
satisfy $X \in R_{\lbrace \infty \rbrace}$, $X^T \in R_{\lbrace
-\infty \rbrace}$ and $Y \in R_{\lbrace -\infty \rbrace}$, $Y^T \in
R_{\lbrace \infty \rbrace}$. Thus, for every pair $(x,y) \in \hatr
\times \hatr$ there exists a matrix $Z$ satisfying $\prjc{Z} =
\lbrace x \rbrace$ and $\prjr{Z} = \prjc{Z^T} = \lbrace y \rbrace$
as required.

Next suppose $M = [x,y]$ is an interval with real endpoints. Then $N
= [w,z]$ must be an interval with real endpoints satisfying $z-w =
y-x$ so that $w+y = x+z$. Now consider the matrix
$$Z = \twomat{0}{w}{x}{w+y} = \twomat{0}{w}{x}{x+z}.$$
Referring once more to Figure~\ref{fig:rclass} we see that $Z \in
R_{[x,y]}$ while $Z^T \in R_{[w,z]}$ so that $\prjc{Z} = M$
and $\prjr{Z} = \prjc{Z^T} = N$ as required.

Now consider the case that $M = [-\infty, y]$ with $y$ real. Then
either $N = [-\infty, z]$ with $z$ real, or $N = [x, \infty]$ with
$x$ real. In the former case it suffices to take the matrix
$$Z = \twomat{0}{z}{y}{-\infty},$$
while in the latter case one considers
$$Z = \twomat{0}{x}{-\infty}{x+y}.$$
In both cases, reference to Figure~\ref{fig:rclass} once more establishes
that the given matrix has the correct column and row spaces.

Finally, an argument entirely similar to the previous one applies in
the case that $M = [y, \infty]$ with $y$ real, and hence completes
the proof.
\end{proof}

\begin{theorem}\label{thm_jorder}
Let $A,B \in \mstwo$. Then the following are equivalent:
\begin{itemize}
\item[(i)] $A \leqj B$;
\item[(ii)] $\prjc{A}$ embeds isometrically in $\prjc{B}$;
\item[(iii)] $\prjr{A}$ embeds isometrically in $\prjr{B}$.
\end{itemize}
\end{theorem}
\begin{proof}
The equivalence of (ii) and (iii) follows from Proposition~\ref{prop_duality}.

Suppose next that (i) holds, and let $X, Y \in \mstwo$ be such that
$A = XBY$. Then $A = XBY \leqr XB$ so by Lemma~\ref{lemma_rclasses},
$\prjc{A} \subseteq \prjc{XB}$, and in particular $\prjc{A}$ embeds
isometrically in $\prjc{XB}$. Similarly, $XB \leql B$ so by
Lemma~\ref{lemma_rclasses}, $\prjr{XB}$ embeds isometrically in
$\prjr{B}$. Now, by Proposition~\ref{prop_duality}, $\prjc{XB} \cong
\prjr{XB}$ and $\prjr{B} \cong \prjc{B}$, so by transitivity of
isometric embedding we conclude that $\prjc{A}$ embeds isometrically
in $\prjc{B}$ and (ii) holds.

Finally, suppose (ii) holds.
Let $M \subseteq \prjc{B}$ be the image of an isometric embedding of
$\prjc{A}$ into $\prjc{B}$. Then $M$ is clearly a closed
convex set isometric to $\prjc{A}$ which by
Proposition~\ref{prop_duality} is also isometric to $\prjr{A}$.
Hence, by Proposition~\ref{prop_intersect}, there is a matrix $Z \in
\mstwo$ such that $\prjc{Z} = M \subseteq \prjc{B}$ and $\prjr{Z} =
\prjr{A}$. But now by Corollary~\ref{cor_rclasses} and
Lemma~\ref{lemma_rclasses} we have $A \mathcal{L} Z$ and $Z \leqr
B$, from which it follows that $A \leqj B$.
\end{proof}

\begin{theorem}\label{thm_dj}
Let $A, B \in \mstwo$. Then the following are equivalent:
\begin{itemize}
\item[(i)] $A \mathcal{D} B$;
\item[(ii)] $A \mathcal{J} B$;
\item[(iii)] $\prjc{A} \cong \prjc{B}$;
\item[(iv)] $\prjr{A} \cong \prjr{B}$.
\end{itemize}
\end{theorem}
\begin{proof}
The equivalence of (iii) and (iv) follows from Proposition~\ref{prop_duality}.
That (i) implies (ii) follows from general facts about semigroups, while
the fact that (ii) implies (iii) is a corollary of Theorem~\ref{thm_jorder}.

Finally, if (iii) holds then by Proposition~\ref{prop_duality} we
have $\prjr{A} \cong \prjc{A} \cong \prjc{B}$, so by
Proposition~\ref{prop_intersect} there is a matrix $Z \in \mstwo$
such that $\prjc{Z} = \prjc{B}$ and $\prjr{Z} = \prjr{A}$. By
Corollary~\ref{cor_rclasses} it follows that $B \mathcal{R} Z$ and
$Z \mathcal{L} A$. Since $\mathcal{D}$ is an equivalence relation
containing $\mathcal{L}$ and $\mathcal{R}$ we conclude that $X
\mathcal{D} Y$ so that (i) holds.
\end{proof}

Theorems~\ref{thm_jorder} and \ref{thm_dj} allows us to deduce a great
deal about the two-sided ideal structure of $\mstwo$. An immediate
corollary is a description of the lattice order on the two-sided principal
ideals (or equivalently, on the $\mathcal{J}$-classes).

\begin{corollary}
The lattice of principal two-sided ideals in $\mstwo$ is isomorphic
to the lattice of isometry types of closed convex subsets of $\hatr$
under the partial order given by isometric embedding.
\end{corollary}

We now turn our attention to non-principal ideals, which it transpires
can also be characterized by convex sets in $\hatr$. Let $\mathbb{S}$ be
the set of convex sets in $\hatr$ consisting of all the closed convex
sets, all the open intervals of finite diameter,
and the open interval $(-\infty, \infty)$. Note that we exclude the
half-infinite open intervals. Once again, it is easily seen
that isometric embedding induces a partial order on the isometry types of
sets in $\mathbb{S}$. Note also that no two isometry types of sets in
$\mathbb{S}$ admit
isometric embeddings of exactly the same collection of closed convex
sets.

\begin{theorem}\label{thm_allideals}
Let $I$ be an ideal of $\mstwo$. Then there exists a subset $I' \in \mathbb{S}$
such that for all $X \in \mstwo$ we have $X \in I$ if
and only if the projective column space of $X$ embeds isometrically
into $I'$. Moreover, the set $I'$ is unique up to isometry.
\end{theorem}
\begin{proof}
Let $I$ be an ideal of $\mstwo$, and let $T$ be the set of all isometry
types of closed convex sets in $\hatr$ which arise as
projective column spaces (or equivalently, projective row spaces) of
matrices in $I$. If $T$ has a maximal element under the isometric embedding
order, then it follows from Theorem~\ref{thm_jorder} that it suffices to take
$I'$ to be this convex set.

Suppose, then, that $T$ has no maximal element. Then clearly it cannot
contain the isometry type of a convex set of infinite diameter (since
there are only finitely many such up to isometry, and they are above all
other convex sets in the isometric embedding order), but must contain
infinitely many intervals of finite diameter. If the diameters of these
intervals are bounded above by a real number, then we let $w$ be the supremum of
the diameters. Since $T$ has no maximal element, this supremum is not
attained in $T$. It follows from Theorem~\ref{thm_jorder} that a matrix lies
in $I$ if and only if its projective column space has diameter strictly less than
$w$. This is the case exactly if the projective column space embeds isometrically
in an open interval of diameter $w$, so it suffices to take $I'$ to be
such an interval.

On the other hand, if the diameters of the intervals are not bounded
above then, by Theorem~\ref{thm_jorder} again, we see that $I$ contains
every matrix with projective row space of finite diameter, and it follows
that we may take $I'$ to be the open interval $(-\infty, \infty)$.

Finally, the uniqueness up to isometry of $I'$ follows from Theorem~\ref{thm_jorder} and
the fact that no two distinct isometry types of sets in $\mathbb{S}$ embed exactly
the same closed convex sets.
\end{proof}

For $I$ an ideal of $\mstwo$, we denote by $S(I)$ the unique convex
subset $S(I) \in \mathbb{S}$ such that $I$ consists of those matrices
with projective column space which embeds isometrically in $S(I)$.

\begin{corollary}\label{cor_totalorder}
The two-sided ideals of $\mstwo$ are totally ordered under inclusion.
\end{corollary}
\begin{proof}
This follows immediately from Theorem~\ref{thm_allideals}, and the obvious
fact that the isometry types of sets in $\mathbb{S}$ are totally ordered under isometric
embedding.
\end{proof}

\begin{corollary}\label{cor_closedprincipalfg}
Let $I$ be an ideal in $\mstwo$. Then the following are equivalent:
\begin{itemize}
\item[(i)] $S(I)$ is closed;
\item[(ii)] $I$ is principal;
\item[(iii)] $I$ is finitely generated.
\end{itemize}
\end{corollary}
\begin{proof}
By Proposition~\ref{prop_intersect}, every closed convex set is the
projective column space of some matrix in $\mstwo$, so that (i) implies
(ii) follows from Theorem~\ref{thm_jorder}.
That (ii) implies (iii) is by definition. Finally, suppose (iii) holds
let $G$ be a finite generating set for $I$, and let $S = \lbrace PC(X) \mid X \in G \rbrace$.
By Corollary~\ref{cor_totalorder}, $S$ is totally ordered under
isometric embedding, and since it is finite, it must contain a maximum
element. This maximum element is a closed convex set, and an easy argument now
shows that it must be equal to $S(I)$.
\end{proof}

The equivalence of (ii) and (iii) in Corollary~\ref{cor_closedprincipalfg}
may be viewed as a manifestation of the fact that every finitely generated
tropical convex cone in $\barr^2$ is closed.

\begin{corollary}
Every ideal in $\mstwo$ is either principal, or the difference between
a principal ideal and its generating $\mathcal{J}$-class.
\end{corollary}
\begin{proof}
Let $I$ be an ideal and consider the convex set $S(I) \in \mathbb{S}$.
If $S(I)$ is
closed then by Corollary~\ref{cor_closedprincipalfg}, $I$ is principal.
Otherwise, $S(I)$ is an open interval. Let $J$ be the the smallest closed
interval in $\hatr$ containing $S(I)$. Clearly, a given closed interval
$K$ embeds isometrically into $S(I)$ if and only if it embeds isometrically
into $J$ but is not isometric to $J$. Hence, by Theorems~\ref{thm_jorder} and
\ref{thm_dj}, a matrix is in $I$ if and only if it lies in the ideal
corresponding to $J$ (which by Corollary~\ref{cor_closedprincipalfg} is
principal) but not in the $\mathcal{J}$-class corresponding to $J$.
\end{proof}


\section{Idempotents and Subgroups}\label{sec_idpt}

Our aim in this section is to identify the idempotent elements of
$\mstwo$, and draw some conclusions about both its
semigroup-theoretic structure and its maximal subgroups.
Recall that an element $e$ in a semigroup is called
\textit{idempotent} if $e^2 = e$.

\begin{proposition}\label{prop_idpt}
The idempotents of $\mstwo$ are exactly the matrices of the form
\[
\left( \begin{array}{c c} 0 & x  \\
y&x+y\\
\end{array}
\right), \;\; \left( \begin{array}{c c} 0 & x  \\
y&0\\
\end{array}
\right),\;\; \left( \begin{array}{c c} x+y & x  \\
y&0\\
\end{array}
\right) \;\; \mbox{and} \;\;
\left( \begin{array}{c c} -\infty & -\infty  \\
-\infty&-\infty\\
\end{array}
\right)
\]
where $x, y \in \bar{\mathbb{R}}$ with $x+y \leq 0$.
\end{proposition}
\begin{proof}
It is readily verified by direct computation that these matrices are
idempotent. Conversely, suppose that
\[
\left( \begin{array}{c c} a & b  \\
c&d\\
\end{array}
\right) \  \left( \begin{array}{c c} a & b  \\
c&d\\
\end{array}
\right) = \left( \begin{array}{c c} a & b  \\
c&d\\
\end{array}
\right).
\]
Then we have
\begin{eqnarray}
\max(a+a, b+c)= a,&&\;\;\max(b+c, d+d)= d,\label{diag}\\
\max(a+c, c+d)= c,&& \;\;\max(a+b, b+d)= b\label{offdiag},
\end{eqnarray}
giving $-\infty \leq a,d \leq 0$.

First suppose that $a < 0$. Then by \eqref{diag} we must have that
$a=b+c$. If $d = 0$ then we have a matrix of the form
$$\left( \begin{array}{c c} b+c & b  \\
c&0\\
\end{array}
\right),$$ where $-\infty \leq b+c = a < 0$. On the other hand, if $d \neq 0$
then by \eqref{diag} we must have the zero matrix.

Next suppose that $a=0$. By \eqref{diag} we see that $b+c \leq 0$
and either $d=0$ or $d=b+c$, giving matrices of the form
$$\left( \begin{array}{c c} 0 & b  \\
c&0\\
\end{array}
\right)\mbox{ and } \left( \begin{array}{c c} 0 & b  \\
c&b+c\\
\end{array}
\right)$$ where $-\infty \leq b+c \leq 0$, respectively.
\end{proof}

While the purely computational approach to finding idempotents employed
in the proof of Proposition~\ref{prop_idpt} is straightforward in the
$2 \times 2$ case, it is conceptually unenlightening and quickly becomes
intractable in higher dimensions.
In any semigroup of functions, the idempotents are exactly the
\textit{projections}, that is, those functions which fix their
images pointwise. In $\mstwo$, then, an idempotent element is a
matrix which (viewed as acting from the left on column vectors)
fixes the tropical convex cone generated by its own columns.
Figure~5 illustrates the geometric action of
some typical idempotents. In higher dimensions, the complex structure
of tropical convex cones \cite{develin} makes it a delicate
task to locate the idempotents by geometric arguments, but nevertheless we
believe that only this approach is feasible.

\begin{figure}[ht] 
\begin{center}\psfrag{a}{u}\psfrag{b}{$u\oplus v$ }\psfrag{c}{$v$}
\psfrag{d}{\hspace{-0.7cm}$\lambda \otimes v$}\psfrag{e}{$v$}
\subfigure[{A projection onto the column space corresponding to
$R_{\lbrace y
\rbrace}$.}]{\includegraphics[scale=0.5]{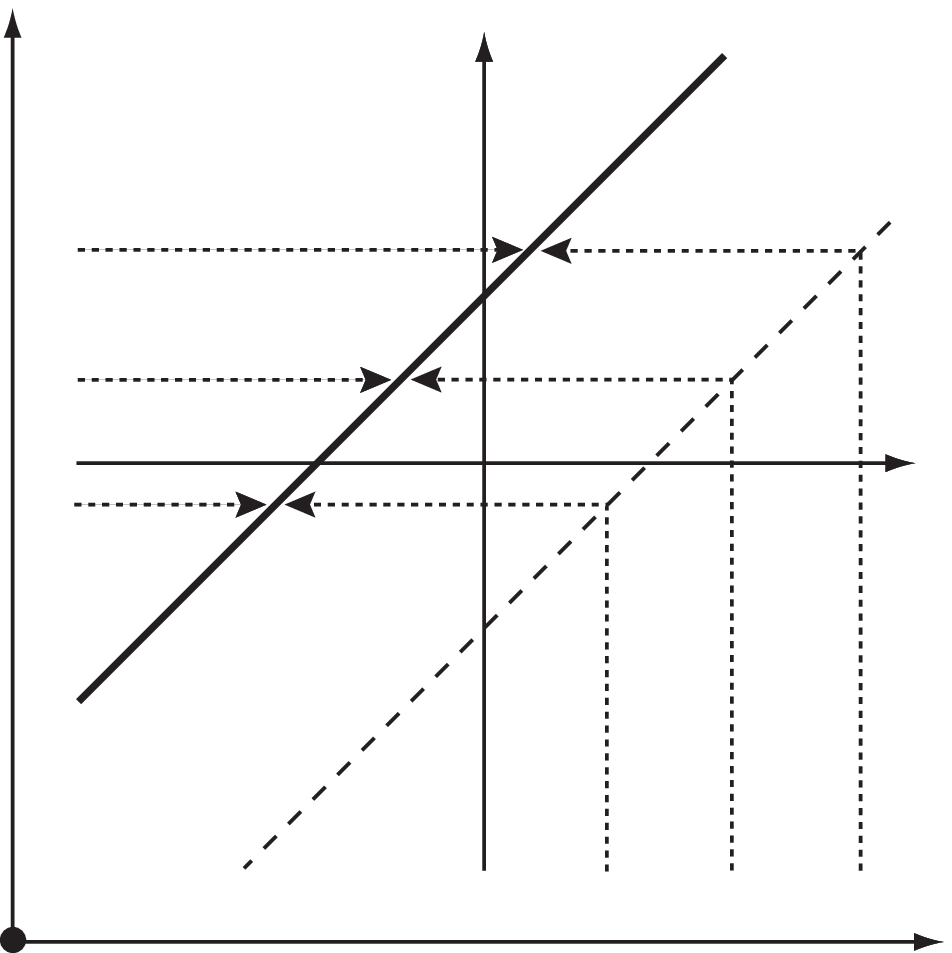}\label{line}}\hspace
{1cm}
\subfigure[{A projection onto the column space corresponding to
$R_{[x,y]}$.}]{\includegraphics[scale=0.5]{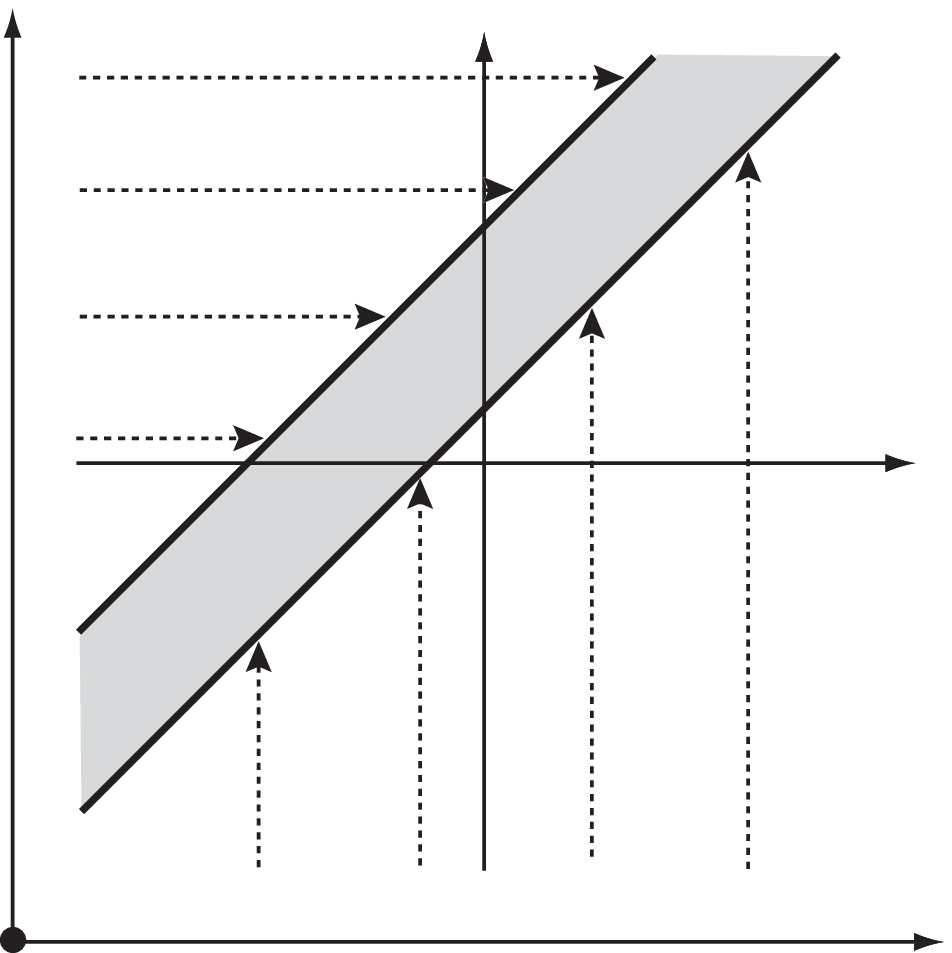}\label{strip}}
\caption{Examples of projections.}\label{projections}\end{center}
\begin{picture}(1,10)
\put(115,202){$y$} \put(95,154){$-x$}\put(283,209){$y$}
\put(283,184){$x$}
\end{picture}
\label{fig:projections}
\end{figure}

Cross-referencing Proposition~\ref{prop_idpt} with
Figure~\ref{fig:rclass}, we quickly see that $\mstwo$ has an idempotent
in every $\mathcal{R}$-class. Recall that a semigroup $S$ is called
\textit{regular} if for every element $X \in S$ there is an element
$Y \in S$ such that $XYX = X$ (\textit{von Neumann regularity} in
the terminology of ring theory). It is well known that a semigroup
is regular if and only if every $\mathcal{R}$-class contains an
idempotent, so we have established the following.

\begin{theorem}
The semigroup $\mstwo$ of all $2 \times 2$ tropical matrices is regular.
\end{theorem}

We now turn our attention to maximal subgroups of $\mstwo$.
It is a foundational result of semigroup theory (see for example \cite{clifford})
that every subgroup of a semigroup lies in a unique maximal subgroup, and that
the maximal subgroups are exactly the $\mathcal{H}$-classes of idempotent
elements. We thus begin by describing those $\mathcal{H}$-classes which
contain idempotents.

\begin{theorem}\label{thm_idpthclasses}
Let $M$ and $N$ be closed convex subsets of $\hatr$. Then the $\mathcal{H}$-class $R_M \cap
L_N$ contains an idempotent if and only if one of the
following conditions holds:
\begin{itemize}
\item[(i)] $M=\{x\}$ and $N=\{y\}$ with $\lbrace x,y \rbrace \neq \lbrace -\infty,\infty \rbrace$;
\item[(ii)] $M = -N = \lbrace -x \mid x \in N \rbrace$ where $|N| \neq1$.
\end{itemize}
\end{theorem}
\begin{proof}
Suppose first that $R_M \cap L_N$ contains an idempotent $E$. Then
$E$ must have one of the four forms given by
Proposition~\ref{prop_idpt}. Clearly if $E$ is the zero matrix then $M =
N = \emptyset$ and (ii) holds. If $E$ has the form
$\twomat{0}{x}{y}{x+y}$ for $x, y \in \barr$ with $x+y \leq 0$, then
it is readily verified that $PC(E) = \lbrace y \rbrace$ and $PR(E) =
\lbrace x \rbrace$ and hence (i) holds. An entirely similar argument
holds if $E$ has the form $\twomat{x+y}{x}{y}{0}$, where this time
$PC(E) = \lbrace -x \rbrace$ and $PR(E) = \lbrace -y \rbrace$.
Finally, if $E$ has the form $\twomat{0}{x}{y}{0}$ with $x+y \leq 0$
then a simple computation shows that $PC(E) = [y,-x]$ and $PR(E) =
[x,-y]$ so that once again (ii) holds.

Conversely, suppose (i) holds, say $M = \lbrace x \rbrace$ and $N =
\lbrace y \rbrace$, where $\lbrace x, y \rbrace \neq \lbrace -\infty, \infty \rbrace$
so that $x+y$ is well-defined. If $x+y \leq 0$
then $x,y \neq \infty$ and the matrix $\twomat{0}{y}{x}{x+y}$ is an
idempotent by Proposition~\ref{prop_idpt}, and is easily seen (by
computing the projective row and column spaces) to lie in the
claimed $\mathcal{H}$-class. On the other hand, if $x+y \geq 0$ then
$x,y \neq -\infty$, so we have $-x, -y \in \barr$ and $(-x) + (-y)
\leq 0$. It follows by Proposition~\ref{prop_idpt} that the matrix
$\twomat{-x-y}{-y}{-x}{0}$ is idempotent and once again it is easily
verified that it lies in $R_M \cap L_N$.

Finally, suppose (ii) holds. If $M$ is empty then so is $N$, and the
zero matrix is an idempotent in $R_M \cap L_N$. Suppose, then, that
$M$ is a closed interval $[x,y]$ with $x,y \in \hatr$ and $x \leq
y$. Then $y \neq -\infty$ so $-y$ is well-defined, and $x+(-y) < 0$.
Hence, by Proposition~\ref{prop_idpt}, the
matrix $\twomat{0}{-y}{x}{0}$ is idempotent. Once more, it is
straightforward to verify that this matrix lies in $R_M \cap L_N$.
\end{proof}

Having ascertained which $\mathcal{H}$-classes are maximal subgroups, it
remains to identify the algebraic structure of each.

\begin{theorem}
The maximal subgroups in the $\mathcal{D}$-class of elements with
row and column space isometric to a closed convex subset $M
\subseteq \hatr$ are isomorphic to:
\begin{itemize}
\item[(i)] the trivial group, if $M = \emptyset$;
\item[(ii)] the additive group $\mathbb{R}$ of real numbers, if $M$ is a
point, or an interval with precisely one real endpoint;
\item[(iii)] the direct product $\mathbb{R} \times S_2$, if $M$ is an interval
with two real endpoints;
\item[(iv)] the wreath product $\mathbb{R} \wr S_2$, if $M = \hatr$.
\end{itemize}
\end{theorem}
\begin{proof}
If $M = \emptyset$ then the only matrix in $R_M \cap L_M$ is the
zero matrix, so this $\mathcal{H}$-class is isomorphic to the
trivial group.

Now suppose $M = \lbrace x \rbrace$ is a singleton. Since maximal
subgroups in a $\mathcal{D}$-class are always isomorphic, by
Theorem~\ref{thm_dj} it suffices to consider the case that $M =
\lbrace -\infty \rbrace$. By Theorem~\ref{thm_idpthclasses}, $R_M
\cap L_M$ contains an idempotent. Reference to
Figure~\ref{fig:rclass} shows that
$$R_M \cap L_M = \lbrace W_a \mid a \in \mathbb{R} \rbrace$$
where
$$W_a = \twomat{a}{-\infty}{-\infty}{-\infty}.$$
Direct calculation shows that $W_a W_b = W_{a+b}$ for all $a, b \in
\mathbb{R}$ so that $R_M \cap L_M$ is isomorphic to the additive
group $\mathbb{R}$ as required.

Next suppose $M = [x,y]$ is an interval with distinct real
endpoints, so that $x-y < 0$. Then by Theorem~\ref{thm_idpthclasses}, setting $N = -M =
[-y,-x]$ we have that the $\mathcal{H}$-class $R_M \cap L_N$
contains an idempotent. A direct computation using Figure~\ref{fig:rclass} shows that
$$R_M \cap L_N = \lbrace X_a, Y_a \mid a \in \mathbb{R} \rbrace$$
where
$$X_a = \twomat{a}{a-y}{a+x}{a} \text{ and } Y_a = \twomat{a}{a-x}{a+y}{a}.$$
Simple calculation, recalling the fact that $x-y < 0$, shows that $X_a X_b = X_{a+b}$, $X_a Y_b = Y_b
X_a = Y_{a+b}$ and $Y_a Y_b = X_{a+b+(y-x)}$ for all $a, b \in
\mathbb{R}$. We deduce that $X_0$ is idempotent and hence is the
identity of $R_M \cap L_N$ and that the $X_a$'s form a central
subgroup isomorphic to the real numbers.
Moreover, choosing $z = (x-y)/2$ we see that
$(Y_z)^2 = X_0$ and every element $Y_b$ can be written in the form
$Y_z X_a$ for some $a \in \mathbb{R}$.
We have shown that $R_M \cap L_N$ is the product of commuting
subgroups with trivial intersection, one of them isomorphic to
$\mathbb{R}$ and the other to $S_2$. It follows that the subgroup is
isomorphic to $\mathbb{R} \times S_2$, as claimed.

Now suppose $M$ is an interval with one real and one infinite
endpoint. By Theorem~\ref{thm_dj} we may assume that $M = [x,
\infty]$. Set $N = -M = [-\infty, -x]$. Then by
Theorem~\ref{thm_idpthclasses} we have that $R_M \cap L_N$ contains
an idempotent. Another reference to
Figure~\ref{fig:rclass} reveals that
$$R_M \cap L_N = \lbrace Z_a \mid a \in \mathbb{R} \rbrace$$
where
$$Z_a = \twomat{a}{-\infty}{a+x}{a}.$$
Once again, we find that $Z_a Z_b = Z_{a+b}$ so that $R_M \cap L_N$
is isomorphic to the additive group $\mathbb{R}$

Finally, if $M = \hatr$ then we also have $N = \hatr$, and $R_M \cap L_N$
is the group of units. We remarked in Section~\ref{sec_prelim} that it is
known that the group of units of $\msn$ is isomorphic to the permutation
group wreath product $\mathbb{R} \wr (S_n, \lbrace 1, \dots, n \rbrace)$. In
the case $n=2$, since the right translation action of $S_2$ on itself is
isomorphic to its standard action on $\lbrace 1, 2 \rbrace$, the group of units is
also isomorphic to the wreath product $\mathbb{R} \wr S_2$ of abstract groups.
\end{proof}

We remarked in Section~\ref{sec_prelim} that it is known that every group
admitting a faithful representation by finite dimensional tropical matrices
has an abelian subgroup of finite index \cite{dalessandro}. In the case of
groups admitting faithful $2 \times 2$ tropical matrix representations, we
can now be rather more precise.
\begin{corollary}
Every group admitting a faithful representation by $2 \times 2$
tropical matrices is either torsion-free abelian or has a torsion-free
abelian subgroup of index $2$.
\end{corollary}
In general, we conjecture that a group admitting a faithful representation
by $n \times n$ tropical matrices must have a torsion free abelian subgroup
of index at most $n!$.

\section*{Acknowledgements}

The authors thank the participants of the
\textit{Manchester Tropical Mathematics Reading Group}, and the organisers,
speakers and participants of the \textit{First de Br\'un Workshop on Computational Algebra}
(held at the National University of Ireland, Galway in 2008) for their
help with learning the background required for this study.
They also thank Claas R\"over for some helpful conversations, and
Gemma Lloyd for her assistance with the diagrams in this paper.

\end{document}